\documentclass{article}[12pt]
\usepackage[dvips]{graphicx}
\usepackage{amsfonts}
\usepackage{amsmath}
\usepackage{amssymb}
\usepackage{amscd}
\usepackage{color}
\usepackage{amsthm}
\usepackage{indentfirst}
\usepackage[hmargin=3cm,vmargin=3cm]{geometry}
\usepackage{parskip}

\newtheorem{thm}{Theorem}
\newtheorem{thm*}{Theorem}

\newtheorem{lma}{Lemma}

\newtheorem{cor}{Corollary}

\theoremstyle{definition}

\theoremstyle{remark}
\newtheorem*{pf}{Proof}

\newtheorem{rmk}{Remark}[subsection]

\newcommand{\R}{{\mathbb{R}}}
\newcommand{\Z}{{\mathbb{Z}}}
\newcommand{\C}{{\mathbb{C}}}

\newcommand{\D}{{\mathbb{D}}}

\newcommand\vol{\operatorname{vol}}
\newcommand\sign{\operatorname{sign}}
\newcommand\lk{\operatorname{lk}}

\newcommand{\G}{\mathcal{G}}

\newcommand{\til}[1]{\widetilde{#1}}

\newcommand{\arr}[1]{\overrightarrow{#1}}

\newcommand{\om}{\omega}
\newcommand{\Om}{\Omega}

\def\Sign{\operatorname{\textbf{Sign}}}

\DeclareMathOperator{\Diff}{\mathrm{Diff}}
\DeclareMathOperator{\Ker}{\mathrm{Ker}}
\DeclareMathOperator{\Cal}{\mathrm{Cal}}

\def\o{\omega}

\begin{document}

\title{\textbf{On the large-scale geometry of the $L^p$-metric on the symplectomorphism group of the two-sphere.}}

\author{\textsc{Michael Brandenbursky and Egor Shelukhin}\\}

\date{}
\maketitle

\begin{abstract}
We prove that the vector space $\R^d$ of any finite dimension $d$ with the standard metric embeds in a bi-Lipschitz way into the group of area-preserving diffeomorphisms $\G$ of the two-sphere endowed with the $L^p$-metric for $p>2$. Along the way we show that the $L^p$-metric on the group $\G$ is unbounded for $p>2$ by elementary methods.
\end{abstract}

\tableofcontents

\section{Introduction and main results}

\subsection{Introduction}

The first results on the large-scale metric geometry of volume-preserving diffeomorphism groups equipped with the hydrodynamic $L^2$-metric were achieved around 1985 by A. Shnirel'man \cite{ShnirelmanGeometry} (cf. \cite{ShnirelmanGeneralized}). He proved that the diameter of the compactly supported diffeomorphism group of the three dimensional cube $I^3=(0,1)^3$ is finite, and conjectured that this diameter is infinite in the two-dimensional case $I^2=(0,1)^2.$ It is a folklore statement \cite{KhesinWendt, ArnoldKhesin, EliashbergRatiu} that Shnirel'man's boundedness result generalizes to arbitrary compact simply-connected Riemannian manifolds (with or without boundary) of dimension $3$ or higher.

Before we survey results pertaining to the question of unboundedness of the $L^p$-metric in the two-dimensional case, we briefly discuss the general method used to establish such a result. Usually, given a diffeomorphism group $\G$ one constructs a function $\Psi:\til{\G} \to \R$, whose absolute value provides a lower bound for the length of a path in $\G$ starting at the identity (up to some additive and multiplicative constants, possibly). This automatically implies that $|\Psi|$ is a lower bound for the corresponding norm on $\til{\G}$. Then, to obtain a result on the norm on $\G$ itself, one argues that $\Psi$ suitably descends to $\G$. This holds for example if $\Psi$ is constant on fibers $pr^{-1}(\phi)$ of the natural projection $pr: \til{\G} \to \G$, up to an error that is bounded uniformly in $\phi \in \G$. It is interesting to identify such a calibrating function in each of the cases considered.

Shnirel'man's conjecture for the two-dimensional case was settled in 1991 by Y. Eliashberg and T. Ratiu \cite{EliashbergRatiu}. In fact they proved two more general results. First the $L^p$-metric on the (identity component) of the group of symplectomorphisms $Symp_{c,0}(M,d\lambda)$ of any compact exact symplectic manifold (necessarily with boundary) is unbounded. Moreover its restriction (as a function of two variables) to the subgroup of Hamiltonian diffeomorphisms $Ham_c(M,d\lambda)$ is unbounded. One can further restrict the metric (as a function of two variables) to the kernel $\Ker(\Cal) \subset Ham_c(M,d\lambda)$ of the Calabi homomorphism \cite{CalabiHomomorphism} $\Cal: Ham_c(M,d\lambda) \to \R$, retaining unboundedness. Note that from this first result it follows as a special case that the $L^p$-metric is unbounded on the group of volume-preserving diffeomorphisms of any compact surface with boundary. The methods of proof of the first result involved the Calabi homomorphism as a key ingredient.

To address the case of closed surfaces, we review Eliashberg and Ratiu's second result. For any compact manifold $M$ with a volume form, whose first Betti number is non-zero and whose fundamental group has trivial center, the diameter of the group of volume-preserving diffeomorphisms endowed with the $L^p$-metric is infinite. From this result it follows that the diameter is infinite for all surfaces of genus $g \geq 2$. The methods of proof of this result involve measuring (on average) the trajectories of the flow against closed differential $1$-forms (hence the condition on the first Betti number).

For the two-torus $T^2$ the situation is slightly more involved. Firstly, we note that since the orbit of any point in $T^2$ under the action of a Hamiltonian loop is contractible (by the existence of the Seidel element \cite{seidel} or by a number of other arguments), one can apply either the original argument of Eliashberg-Ratiu or \cite[Theorem 1.2]{BrandenburskyKedra2} to prove that the $L^p$-metric on $Ham(T^2,dx \wedge dy)$ is unbounded. Then it is rather uncomplicated to upgrade this result to the unboundedness of the $L^p$-metric on $\Diff_0(T^2,dx\wedge dy)$. Alternatively, one can use a folklore argument measuring the displacement of a ball in the universal cover, to prove said unboundedness (this argument, related to us by L. Polterovich, resembles the proof of \cite[Part II, Lemma 5.7]{LalondeMcDuff12}).


Therefore the question that remained open was the case of the two-sphere $S^2$, which is neither exact nor has first cohomology. We note that remarkable progress was achieved in the $1990$'s and the $2000$'s in the study of a different metric on the group of Hamiltonian diffeomorphisms, introduced by H. Hofer \cite{HoferMetric}. We refer to \cite{LalondeMcDuffEnergy, LalondeMcDuff12, PolterovichS2, PolterovichBookGeometry} for several early results in this direction, and shall not survey this intensive area of research fully, for the sheer number of developments and for reasons of relevance to the issue at hand in this paper. We note however, that a result of L. Polterovich from $1998$ \cite{PolterovichS2} on the unboundedness of the Hofer metric on $Ham(S^2,dV)$ implies by the Sobolev inequality the unboundedness of the $L^p$-metric in the case of $S^2$ for exponents $p > 2$. However, since the Sobolev inequality in dimension two fails at exponent $p=2$, one does not deduce the unboundedness of the hydrodynamic $L^2$-metric from Polterovich's result. In this paper we reprove the unboundedness of the $L^p$-metric for $p>2$ in an elementary way.

The main result of this paper, generalizing unboundedness of the $L^p$-metric for $p>2,$ is the existence of bi-Lipschitz embeddings of normed vector spaces of arbitrary finite dimension into the group of symplectomorphisms of $S^2$ endowed with the $L^p$-metric for $p>2$. Our methods are related to a different line of research in two-dimensional Hamiltonian dynamics than described above, and have to do with braiding and relative rotation numbers of trajectories in extended phase space of time-dependent two-dimensional Hamiltonian flows. We note that Shnirel'man has proposed to use relative rotation numbers to bound from below the $L^2$-lengths of two-dimensional Hamiltonian paths in \cite{ShnirelmanGeneralized}. This direction is related to the method of Eliashberg and Ratiu by a theorem of J.-M. Gambaudo and E. Ghys \cite{GambaudoGhysEnlacements}. Gambaudo and Ghys proved that up to a multiplicative constant, the Calabi homomorphism is proportional to the relative rotation number of the trajectories of two distinct points in the two-disc $\D$ under a Hamiltonian flow, averaged over the configuration space of ordered pairs of distinct points $(z_1,z_2)$ in the two-disc (we will call this the average rotation number of the flow). While by virtue of this theorem it follows immediately that the average rotation number is a lower bound for the $L^p$-distance from the identity (since the Calabi invariant is), it is instructive to obtain this bound directly. Such an estimate was performed by J.-M. Gambaudo and M. Lagrange \cite{GambaudoLagrange} by a creative use of the H$\ddot{\text{o}}$lder inequality, to obtain a different proof for the unboundedness of the $L^p$-metric for the two-disc (a modification of which applies for $\Ker(\Cal)$ too). Their argument was later pushed further by M. Benaim and J.-M. Gambaudo \cite{BenaimGambaudo} to obtain quasi-isometric embeddings into
$\Ker(\Cal) \subset \Diff_c(\D,dx\wedge dy)$ of free groups on an arbitrary finite number of generators and of $\Z^d$ for all finite dimensions $d$.

To explain our methods we should explain the functions we use to calibrate the $L^p$-norm. These are quasimorphisms - functions that are additive with respect to the group action - up to an error which is uniformly bounded (as a function of two variables), on groups of area-preserving diffeomorphisms. The quasimorphisms we use were introduced and studied by Gambaudo and Ghys in a beautiful paper from 2004 \cite{GambaudoGhysCommutators}. These quasimorphisms essentially appear from invariants of braids traced out by the action a Hamiltonian path on an ordered $n$-tuple of distinct points in the surface (the traces should be suitably closed up to produce pure braids, and the resulting invariants averaged over the configuration space of $n$-tuples of distinct points). The general procedure for constructing quasimorphisms on volume-preserving diffeomorphism groups from quasimorphisms on the fundamental group (this is the case of the configuration space of 1-tuples) was described by Polterovich in \cite{PolterovichDynamicsGroups}, and quasimorphisms on diffeomorphism groups appearing from invariants of braids and related constructions were further studied in
\cite{PyTorus, PySurfacesQm, BrandenburskyKnots} and in other works.

In \cite{BrandenburskyLpMetrics} it was shown that for the case of $1$-tuples and for the case of $n$-tuples of points on the standard disc, the above quasimorphisms calibrate the $L^p$-norm, improving in particular the result of Benaim and Gambaudo to bi-Lipschitz embeddings of $\R^d$ with the $l^1$-metric into $\Ker(\Cal)$, and sharpening the results of Eliashberg-Ratiu (see also \cite{BrandenburskyKedra1, BrandenburskyKedra2}). In this paper we produce similar estimates for the case of the two-sphere. Technically, our case of $\Diff(S^2,dV)$ is more difficult than that of
${\Ker(\Cal) \subset \Diff_c(\D,dx\wedge dy)}$ because the required analytical and topological bounds require a more global approach and have to take into account the geometry of the sphere.

As in \cite{BrandenburskyLpMetrics} for the case of the two disc, our methods readily give a stronger result than unboundedness. Namely, we prove that for every $p>2$ the group $\Diff(S^2,dV)$ with the $L^p$-metric contains bi-Lipschitz embedded vector spaces $\R^d$ of arbitrary finite dimension $d$ (equipped with the standard metric). We note that the natural counterpart of this statement in Hofer geometry is completely open: it is not known whether or not $Ham(S^2,dV)$ equipped with the Hofer metric is quasi-isometric to $\R$.



\subsection{Preliminaries}
\subsubsection{The $L^p$-metric}
Let $M$ denote a compact connected oriented Riemannian manifold (possibly with boundary) with a volume form $\mu$. We denote by $\G=\Diff_{c,0}(M,\mu)$ the identity component of the group of diffeomorphisms of $M$ preserving $\mu$, that are $Id$ near the boundary if $\partial M \neq 0$. Alternatively, in the case of non-empty boundary, one can consider the open manifold $M \setminus \partial M$ and take compactly supported diffeomorphisms preserving $\mu$.

Given a path $\{\phi_t\}$ in $\G$ between $\phi_0$ and $\phi_1$, we define its $l^p$-length by
\[l_p(\{\phi_t\}) = \int_0^1 dt \, (\int_M |X_t|^p \mu)^{\frac{1}{p}},\]
where $X_t = \frac{d}{dt'}|_{t'=t} \phi_{t'} \circ \phi_t^{-1}$ is the time-dependent vector field generating the path $\{\phi_t\}$, and $|X_t|$ its length with respect to the Riemannian structure on $M$. As is easily seen by a displacement argument, this length functional determines a non-degenerate metric on $\G$ by the formula
\[d_p(\phi_0,\phi_1) = \inf \; l_p(\{\phi_t \}),\]
where the infimum runs over all paths $\{\phi_t\}$ in $\G$ between $\phi_0$ and $\phi_1$. It is immediate that this metric is right-invariant. We denote the corresponding norm on the group by
\[||\phi||_p = d_p(Id,\phi).\]
Clearly $d_p(\phi_0,\phi_1) = ||\phi_1\phi_0^{-1}||_p$. Similarly one has the $L^p$-norm on the universal cover $\til{\G}$ of $\G$, defined for $\til{\phi} \in \til{\G}$ as
\[||\til{\phi}||_p = \inf \; l_p(\{\phi_t\}),\]
where the infimum is taken over all paths $\{\phi_t\}$ in the class of $\til{\phi}$.

We note that up to bi-Lipschitz equivalence of metrics ($d$ and $d'$ are equivalent if $\frac{1}{C}d \leq d' \leq C d$ for a certain constant $C>0$) the $L^p$-metrics on $\G$ and on $\til{\G}$ are independent of the choice Riemannian structure and of the volume form $\mu$ compatible with the orientation on $M.$ In particular, the question of boundedness or unboundedness of the $L^p$-metric enjoys the same invariance property.

\subsubsection{Quasimorphisms}
The notion of a quasimorphism will play a key role in our arguments. Quasimorphisms are helpful tools for the study of non-abelian groups, especially those that admit few or no homomorphisms to the reals. A quasimorphism $r: G \to \R$ on a group $G$ is a real-valued function that satisfies
\[r(xy) = r(x) + r(y) + b(x,y),\]
for a function $b:G\times G \to \R$ that is uniformly bounded:
\[\delta(r): = \sup_{G\times G} |b| < \infty.\]

A quasimorphism $\overline{r}:G \to \R$ is called \textit{homogeneous} if $\overline{r}(x^k) = k \overline{r}(x)$ for all $x\in G$ and $k \in \Z$. To any quasimorphism $r:G\to \R$ there corresponds a unique homogeneous quasimorphism $\overline{r}$ that differs from $r$ by a bounded function:
\[\sup_G |\overline{r} - r| < \infty.\]
It is called the \textit{homogenization} of $r$ and satisfies
\[\overline{r}(x) = \lim_{n \to \infty} \frac{r(x^n)}{n}.\]

We refer to \cite{CalegariScl} for more information about quasimorphisms.

\subsection{Main results}

We consider the following situation. Our manifold is $M=S^2=\C P^1$ endowed with the Fubini-Study symplectic form $dV$ scaled to have total volume $2\pi$. Note that this is twice the Riemannian volume of the standard Fubini-Study metric on $S^2$, which is the Riemannian metric that we equip $S^2$ with. For a given integer $n$ we have the associated configuration space $X=X_n(M)$ of ordered $n$-tuples of distinct points in $M$. This space can be considered as the complement in the complex manifold $M^n=M \times ... \times M$ of a union of $n(n-1)/2$ complex hypersurfaces. For $1\leq i < j \leq n$ such a hypersurface $D_{ij}$ is defined by the equation $z_i = z_j$ for $z_i \in M, z_j \in M$. Note that the Fubini-Study Riemannian measure induces a measure $\nu$ of finite volume on $X$.

Given a real valued quasimorphism $r$ on the fundamental group $\Gamma=P_n(M) = \pi_1(X_n(M),m)$ for a fixed base-point $m \in X_n(M)$ there is a natural way to induce a real valued quasimorphism on the universal cover $\widetilde{\G}$ of the group $\G$ of all volume preserving diffeomoprhisms of $M=S^2$. Since in our case of $M=S^2$ the fundamental group of $\G$ is finite, this induces a quasimorphism on $\G$ itself. The construction is carried out by the following steps (cf. \cite{GambaudoGhysCommutators,PolterovichDynamicsGroups,BrandenburskyKnots}).

\begin{enumerate}
\item For all $x \in X'=X \setminus Z$, with $Z$ a closed negligible subset (e.g. a union of submanifolds of positive codimension) choose a smooth path $\gamma(x):[0,1] \to X$ between the basepoint $m\in X$ and $x$. Make this choice continuous in $X'$. We perform this choice by choosing such a system of paths on $M$ itself and then considering the induced coordinate-wise paths in $M^n$, and picking $Z$ to ensure that these induced paths actually lie in $X$. After choosing the system of paths $\{\gamma(x)\}_{x \in X\setminus Z}$ we extend it measurably to $X$ (obviously, no numerical values computed in the sequel will depend on this extension). We call the resulting choice a system of "short paths". On the two-sphere, we use minimal geodesics as paths on the manifold itself. We note, however, that the resulting quasimorphism will not depend, up to bounded error, on the choice of the system of short paths - as long as the new system of paths differs from the old one by a system of loops whose image in $\pi_1(X,m)$ is bounded. For the calculations that follow we will make a choice of short paths that is equivalent in this manner to the one induced by minimal geodesics. We shall specify it later.
\item Given a path $\{\phi_t \}_{t \in [0,1]}$ in $\G$ starting at $Id$, and a point $x \in X$ consider the path $\{\phi_t \cdot x\}$, to which we then catenate the corresponding short paths. That is consider the loop
    \[l(x) := \gamma(x) \# \{\phi_t \cdot x\} \# \gamma(y)^{-1} \in \Om_m X,\]
     where $^{-1}$ denotes time reversal. Hence we obtain for each $x \in X$ an element $[l(x)] \in \pi_1(X,m)$. Consequently applying the quasimorphism $r:\pi_1(X,m) \to \R$ we obtain a measurable function $g:X \to \R$. Namely $g(x) = r([l(x)])$. The quasimorphism $\Phi$ on $\widetilde{\G}$ is defined by
     \[\Phi([\{\phi_t\}]) = \int_X g \,d \nu.\]
     It is immediate to see that this function is well-defined by topological reasons. The quasimorphism property follows by the quasimorphism property of $r$ combined with finiteness of volume. The fact that the function $g$ is absolutely integrable can be shown to hold a-priori as in Appendix \ref{Appendix: Integrability}. We note that by Tonelli's theorem it also follows as a by-product of the proof of our main theorem.
\item Of course our quasimorphism can be homogenized, to obtain a homogeneous quasimorphism $\overline{\Phi}$.
\end{enumerate}
\vspace{2mm}
\begin{rmk}
In our case, by the result of Smale \cite{Smale} $\pi_1(\G) = \Z/2\Z$, and hence the quasimorphisms descend to quasimorphisms on $\G$, e.g. by minimizing over the two-element fibers of the projection $\widetilde{\G} \to \G$. For $\overline{\Phi}$, the situation is easier since it vanishes on
$\pi_1 (\G)$, and therefore depends only on the image in $\G$ of an element in $\til{\G}$. We keep the same notations for the induced quasimorphisms.
\end{rmk}


It is well-known (cf. \cite[Chapter 4]{BirmanBook}) that the fundamental group $\Gamma$ is generated by the homotopy classes of small loops $\gamma_{ij}$ around the partial diagonals $D_{ij}.$ We estimate the above quasimorphisms using the standard affine chart $u_0: \C \xrightarrow{\sim} U_0 \subset \C P^1, z \mapsto [z,1]$ of full measure. Note that the generators $\gamma_{ij}$ can be chosen to lie in the image on $\pi_1$ of the composition of canonical maps \[X_n(\C) \xrightarrow{\sim} X_n(U_0) := X_n(\C P^1) \cap U_0 \times ... \times U_0 \to X_n(\C P^1),\] the intersection taken inside $(\C P^1)^n.$

By a theorem of Arnol'd \cite{ArnoldColoredBraids} the first cohomology of $X_n(\C)$ is generated by $(-i)$ times the cohomology classes of closed $1$-forms $\alpha_{ij},$ $1\leq i < j \leq n$, satisfying $\int_{\gamma_{ij}} \alpha_{i'j'} = i \delta^{i'}_i \delta^{j'}_j$.

Such $\alpha_{ij}$ can be constructed as the restriction to $X_n(\C)$ of the pull-back $p_{ij}^* \alpha,$ for
$$p_{ij}: \C^n\setminus D_{ij} \to \C^2 \setminus D$$
the natural projection ($D$ being the diagonal in $\C^2$) and
\[\displaystyle{\alpha = \frac{1}{2\pi}\frac{d(a- z)}{a-z}}\] for coordinates $(a,z)$ on $\C \times \C,$
which correspond in our chosen chart to homogenous coordinates $([a,1],[z,1])$ on $\C P^1 \times \C P^1$. We set $\theta_{ij} := Im(\alpha_{ij})$.

For a $1$-form $\theta$ on a manifold $Y$ and a smooth parameterized path $\gamma:[0,1] \to Y$ set
\[\int_\gamma |\theta| := \int_0^1 |\theta_{\gamma(t)}(\dot{\gamma}(t))| dt.\]
Clearly, for a smooth loop $\gamma$ we have $|\int_\gamma \theta | \leq \int_\gamma |\theta |$.
Note moreover that for any loop $\gamma'$ homologous to $\gamma$ we similarly have $|\int_\gamma \theta | \leq \int_{\gamma'} |\theta |$.
Moreover, $\int_{\gamma} |\theta| = \int_{\gamma^{-1}} |\theta|$, where $\gamma^{-1}$ is the time-reversal of $\gamma$.

Our first theorem is the following estimate:
\vspace{2mm}
\begin{thm}\label{Theorem: estimate}
For every $p > 2$ the quasimorphism $\Phi$ is controlled from above by the $L^p$-norm. Namely there exists a constant $C>0$ such that for each path $\arr{\phi}=\{\phi_t\}$ in $\G$ starting at $Id$, we have
\[|\Phi([\arr{\phi}])| \leq C \cdot (l_p(\arr{\phi})+1).\]
\end{thm}

It follows immediately that
\vspace{2mm}
\begin{cor} There exists a constant $C'>0$ such that
\[|\Phi(-)| \leq C' \cdot (||-||_p+1),\]
for all $p>2$.
\end{cor}

By averaging and the triangle inequality for norms we have
\vspace{2mm}
\begin{cor} For all $p>2$ we have
\[|\overline{\Phi}(-)| \leq C' \cdot ||-||_p\hspace{1mm}.\]
\end{cor}

By choosing any such quasimorphism $\Phi$ that is unbounded (equivalently, non-zero such $\overline{\Phi}$) we obtain the following (c.f. \cite{BrandenburskyLpMetrics}).
\vspace{2mm}
\begin{cor}\label{Corollary: unbounded}
The metrics $d_p$ on $\til{\G}$ and on $\G$ are unbounded for $p>2$.
\end{cor}

By choosing a sequence $\overline{\Phi}_1,...,\overline{\Phi}_d$ of homogeneous quasimorphisms satisfying $\overline{\Phi}_i(\phi_j) = \delta_{ij}$ for a sequence $\phi_1,...,\phi_n$ of commuting diffeomorphisms in $\G$ we enhance this to the main result of this paper.
\vspace{2mm}
\begin{cor}\label{Corollary: bi-Lipschitz lattices}
For each $p>2$ and $d\geq 1$, there is a bi-Lipschitz embedding $\R^d \to \G$ of groups, where the real vector space $\R^d$ is endowed with the metric coming from the $l^1$-norm, and $\G$ is endowed with the $L^p$-metric.
\end{cor}

See Section \ref{Section: examples} for examples of such choices.

\subsection*{Acknowledgements.}
We thank Leonid Polterovich for illuminating discussions. We thank Lev Buhovsky, Yakov Eliashberg and Alexander Shnirel'man for useful conversations.
E.S. thanks Octav Cornea for his support and Steven Lu for the invitation to speak on a preliminary version of these results. Both authors thank CRM-ISM Montreal for a great research atmosphere. We thank the anonymous referee for bringing a mistake in the previous version of this manuscript to our attention.

\section{Proof.}


The proof consists of two steps: the topological bound and the analytical bounds, the details of which appear in subsequent sections.
For our purposes it is convenient to consider the generating set $S =\bigcup_{i,j}\{[\gamma_{ij}],[\gamma_{ij}]^{-1}\}$ in the spherical pure braid group $P_n(S^2) = \pi_1(X_n(S^2),m)$. We denote by $|[\gamma]|_S$ the corresponding word-length of an element $[\gamma]$ in $P_n(S^2)$.

%



Our goal is to estimate $|\Phi([\arr{\phi}])|$ for a path $\arr{\phi}=\{\phi_t\}$ in $\G$ starting at $Id$. We first estimate the integrand
$g(x) = r([l(x)])$. By the quasimorphism property, for $C_1 = \max_{\gamma \in S} |r(\gamma)| + \delta(r)$ (recall that $\delta(r)$ denotes the defect $\sup_{x,y \in \pi_1(X)}|r(xy)-r(x)-r(y)| < \infty$ of the homogeneous quasimorphism $r$) we get
\begin{equation}\label{Equation: estimating the integrand via word length}|g(x)|\leq C_1 \cdot |[l(x)]|_S.\end{equation}
Therefore it is enough to estimate
\[ \int_X d\nu(x) \; |[l(x)]|_S.\]
In order to do this we prove the following topological bound. We note that for all $x$ outside a subset of a measure zero subset in $X_n(M)$, the loop $l(x)$ lies in $X_n(U_0).$
\vspace{2mm}
\begin{lma}\label{Lemma: topological bound}
For a loop $l(x)$ in $X_n(U_0),$
\[|[l(x)]|_S \leq 2 A \sum_{i<j} \int_{\l(x)} |\theta_{ij}| + B,\]
for some constants $A>0, B>0$.
\end{lma}

Hence, it is enough to estimate
\begin{equation}\label{Equation: expression l(x)}
\int_X d\nu \int_{l(x)} |\theta_{ij}|,
\end{equation}
the integrand being almost everywhere defined.

Note that
\[\int_{l(x)} |\theta_{ij}| = \int_{\gamma(x)}|\theta_{ij}| +  \int_{\gamma(y)}|\theta_{ij}| + \int_{\{\phi_t \cdot x \}} |\theta_{ij}|.\]
We show in Lemma \ref{Lemma: bound on angles, nice short paths} that for an appropriate choice of short paths we have $\int_{\gamma(q)}|\theta_{ij}| \leq C_3$ for all $q \in X$. This will reduce estimating the expression in Equation \ref{Equation: expression l(x)} to estimating \begin{equation}\label{Equation: expression phi_t cdot x via theta}
\int_X d\nu \int_{\{\phi_t \cdot x\}} |\theta_{ij}|.
\end{equation}
We note that $|\theta_{ij}| \leq |\alpha_{ij}|$, where in the right hand side we use the standard norm on complex numbers, and
hence it is sufficient to estimate
\begin{equation}\label{Equation: expression phi_t cdot x via alpha}
\int_X d\nu \int_{\{\phi_t \cdot x\}} |\alpha_{ij}|.
\end{equation}

To this end we have the following
\vspace{2mm}
\begin{lma}\label{Lemma: estimate integral of integrals over paths} For $p>2$ we have
\[\int_X d\nu \int_{\{\phi_t \cdot x\}} |\alpha_{ij}| \leq C_4 \cdot l_p(\arr{\phi})\hspace{1mm}.\]
\end{lma}

Hence, combining estimate (\ref{Equation: estimating the integrand via word length}), Lemma \ref{Lemma: topological bound}, Lemma \ref{Lemma: estimate integral of integrals over paths} and Lemma \ref{Lemma: bound on angles, nice short paths} gives a proof of Theorem \ref{Theorem: estimate}.

\section{Topological bound}\label{Section: topological bound}

In this section we prove Lemma \ref{Lemma: topological bound}. The idea is to use a full-measure chart on $\C P^1$ diffeomorphic to $\C$, where the topological bounds can be applied.

We shall estimate the expression $|[l(x)]|_{S}$ via $\sum_{i<j} \int_{\l(x)} |\theta_{ij}|$. Indeed, consider the divisors $D^\infty_k$ that are defined by $z_k = x_\infty = [1,0]$, in homogenous coordinates on $\C P^1$. Denote by $D^\infty$ their union (sum). Consider the trace $D^\infty_{\arr{\phi}}=\bigcup_{t \in [0,1]} \phi_t \cdot (D^\infty)$. Since this can be considered as a smooth cycle with boundary whose dimension is one less that of $X$, its measure is zero. Denoting $U_0 := \C P^1 \setminus \{[1,0]\},$ consider the subset $X''=X\setminus D^\infty_{\arr{\phi}} \subset (U_0 \times ... \times U_0)\cap X \subset X$ of full measure. It is the maximal subset of $X$ with the property that $x \in X \setminus D^\infty$ if and only if $\phi_t \cdot x \in X\setminus D^\infty$ for all $t \in [0,1]$. Note that $(U_0 \times ... \times U_0)\cap X \cong X_n(\C)$ canonically, hence for the length of the current estimate we shall assume that all points of $S^2$ that are considered lie in $\C$. Firstly
\[ \int_X d\nu(x) \; |[l(x)]|_S = \int_{X''} d\nu(x) \; |[l(x)]|_S.\]
We now estimate the right hand side of this equality. Note that $\theta_{ij} = p_{ij}^* \theta$, and the latter takes the form $Im(\frac{d(a-z)}{a-z})$ in the chart $U_0 \times U_0$. Hence for $x \in X''$, we have $\int_{l(x)} |\theta_{ij}| = \int_{l(x)_{ij}} |\theta|$, where
$l(x)_{ij} = p_{ij}\circ l(x)$, and moreover the following equality holds by the co-area formula.

For almost all $\omega \in S^1$, the quantity
$$n_{ij}(\omega) = \#\{ t \in [0,1) | \frac{p_i\circ l(x)(t) - p_j \circ l(x)(t)}{|p_i\circ l(x)(t) - p_j \circ l(x)(t)|} = \om \in S^1\}$$
is finite (and well-defined). Moreover, \[\int_{l(x)} |\theta_{ij}| = \int_{S^1} n_{ij}(\omega) dm(\om)\]
for $m$ the Haar (Lebesgue) measure on $S^1$. Note that (cf. \cite{BrandenburskyLpMetrics, BrandenburskyKedra1}) $n_{ij}(\omega)$ is the number of times that the $i$-th strand overcrosses the $j$-th strand in the diagram of the braid $l(x)$ obtained by projection in the direction $\om$. From the $\om$-projection diagram of the braid $l(x)$ we get a presentation of $l(x)$ as a word in the full braid group $B_n(\C)$, generated by say the half-twists, that has exactly one generator for each overcrossing.
Hence
\[\sum_{i \neq j} n_{ij}(\omega) \geq |[l(x)]|_{B_n(\C)}.\]
Consequently
\[2 \int_{l(x)} \sum_{i<j} |\theta_{ij}| = \int_{S^1}\sum_{i \neq j} n_{ij}(\omega) dm(\om) \geq |[l(x)]|_{B_n(\C)}.\]
By standard geometric group theory (cf. \cite[Corollary 24]{TopicsGGT}), since the pure braid group $P_n(\C)$ is a subgroup of finite index in $B_n(\C)$,  we have
\[|\cdot|_{P_n(\C)} \leq A |\cdot|_{B_n(\C)} + B,\]
for some $A >0, B>0$.
Since the small loops around the diagonals whose classes generate the pure braid group $P_n(S^2)$ can be taken to be disjoint from the infinity divisor $D^\infty$, we see that this generating set is actually an image of a generating set of $P_n(\C)$ under the natural surjection $i_*$ of fundamental groups induced by the inclusion. Therefore
\[|i_*(\cdot)|_{P_n(S^2)} \leq |\cdot|_{P_n(\C)}\]
with respect to the chosen generating set. That is
\[|\gamma|_{P_n(S^2)} \leq \min_{\widetilde{\gamma} \in (i_*)^{-1}(\gamma)} |\widetilde{\gamma}|_{P_n(\C)}.\]
Combining the last three observations we have
\[|[l(x)]|_{P_n(S^2)} \leq 2A \int_{l(x)} \sum_{i<j} |\theta_{ij}| + B.\]
This proves the lemma.

\section{Analytical bounds}\label{Section: analytical bounds}

\subsection{Estimating middle term}

We first prove Lemma \ref{Lemma: estimate integral of integrals over paths}. We adapt an estimate of Gambaudo-Lagrange (c.f. \cite{BrandenburskyLpMetrics}), relying on a creative use of the H$\ddot{\text{o}}$lder inequality, taking into account the conversion factors between the spherical and the Euclidean metrics and volume forms.
\vspace{2mm}
\begin{pf}

Since $\alpha_{ij} = p_{ij}^* \alpha$, and since projections in Banach spaces have norm $1$, it is enough to show the estimate for the case $n=2$, namely $X = \C P^1 \times \C P^1 \setminus D$. For the path $\{\phi_t\}$ in $\G$ we have
\[\int_{X} d\nu(x) \int_{\{\phi_t \cdot x\}} |\alpha| =  \int_{X} d\nu(x) \int_0^1 |\alpha_{\phi_t \cdot x}(X^2_t(\phi_t \cdot x))|dt .\]
Here $X^2_t = X_t \oplus X_t$ is the time-dependent vector field coming from the diagonal action of $\G$ on $M^2$. Note that this action preserves the measure $\nu$. If necessary, we restrict it to $X$ without changing notation.
Therefore
\[ \int_{X} d\nu(x) \int_0^1 |\alpha_{\phi_t \cdot x}(X^2_t(\phi_t \cdot x))|dt = \int_0^1 dt \int_X d\nu(x) |\iota_{X^2_t}\alpha|,\]
by Tonelli's theorem. Of course by definition of the Lebesgue integral, we may integrate over $M^2$ extending the integrand by $+\infty$ to the diagonal.




To estimate this integral in terms of the $L^p$-length we proceed as follows.
Denote by $u_{00}$ the chart map
$$u_{00}= u_0 \times u_0: \C \times \C \; \xrightarrow{\sim} \; U_0 \times U_0 \subset \C P^1 \times \C P^1.$$


Note that $2\pi \alpha = \frac{d(a-z)}{a-z},$ and $u_{00}^* X^2_t = (A_t,Z_t)$ for vector fields
$A_t = u_0^* X_t ,Z_t = u_0^* X_t$ on $\C$. It is easy to see that via the affine trivialization of the tangent bundle to $\C \times \C$,  $A$ depends only on the coordinate $a$ and $Z$ only on $z$. Moreover, $u_{00}^* \nu = \mu \boxtimes \mu$, for $\mu$ the pullback by the standard affine chart of our chosen Riemannian measure $dV$ on $\C P^1$. Therefore
\[\int_0^1 dt \int_{X_2(U_0)} d\nu \; |\iota_{X^2_t}\alpha| = \int_0^1 dt \int_{\C \times \C \setminus \Delta} \frac{|A_t(a)- Z_t(z)|}{|z-a|} \;d\mu(a) \; d\mu(z). \]

We apply the triangle inequality $|A_t(a) - Z_t(z)| \leq |A_t(b)| + |Z_t(z)|$, and estimate the two resulting terms separately. Since the two summands are estimated analogously, we show the case of the first summand.

Recall the conversions for the pullback of the spherical measure
\[d\mu(\zeta) = 2 (1+|\zeta|^2)^{-2} \, dm(\zeta)\]
and for the pullback $|\cdot|_{Sph} = u_{0}^* |\cdot|_{S^2}$, of the metric on the sphere in our coordinate chart \[(|\cdot|_{Sph})_{\zeta} = (1+|\zeta|^2)^{-1} \, |\cdot|_{Eucl},\] where $m$ denotes the Lebesgue measure on $\C$ and $|-|_{Eucl}$ is the Euclidean metric.

We have
\[\int_0^1 dt \int_\C d{\mu}(a) \int_\C d{\mu}(z) \frac{|A_t(a)|}{|a-z|} = \int_0^1 dt \int_\C |A_t(a)|_{Sph}\; d{\mu}(a)  \int_\C d{\mu}(z) \frac{(1+|a|^2)}{|z-a|} \]

\[ \displaystyle{\leq \int_0^1 dt \;(\int_\C |A_t(a)|_{Sph}^p \; d{\mu}(a))^{\frac{1}{p}} \, \cdot \,(\int_\C  \, (\int_\C d{\mu}(z) \frac{(1+|a|^2)}{|z-a|})^q  \; d{\mu}(a))^{\frac{1}{q}},}\]
applying the H$\ddot{\text{o}}$lder inequality to the functions $\phi(a)=|A_t(a)|$ and $\psi(a)=(1+|a|^2) \int_\D d{\mu}(z) \frac{1}{|z-a|}$, and the measure ${\mu}.$


The integral $\psi$ takes the form \[\psi(a) = 2 (|a|^2+1) \int_{\C} \frac{1}{|z-a|}\frac{1}{(1+|z|^2)^2} dm(z).\] This expression obeys the following bound, the proof of which we defer to the end of this section.

\vspace{2mm}
\begin{lma}\label{Lemma: estimate an integral}
There exists a universal constant $C_*>0$ such that $\psi(a) \leq C_*(|a|^2+1)^{1/2}$.
\end{lma}

Consequently
\[ \int_0^1 dt \int_\C d{\mu}(a) \int_\C d{\mu}(z) \frac{|A_t(a)|}{|z-a|} \leq C_* \cdot \displaystyle{\int_0^1 dt \;(\int_\C |A_t(a)|_{Sph}^p \; d{\mu}(a))^{\frac{1}{p}} (\int_{\C} (1+|a|^2)^{q/2} d\mu(a))^{\frac{1}{q}}.}\]

The integral \[\int_{\C} (1+|a|^2)^{q/2} d\mu(a) = 2\int_{\C} (1+|a|^2)^{(q-4)/2} dm(a) = C_p\] converges if and only if $q<2$, or equivalently $p>2$. We note that this covers the case $p=\infty$.

Therefore, for $p>2$ we obtain the bound

\begin{equation}\label{Equation: bound phi_t cdot x part via l_p finite}
\int_0^1 dt \int_{X} d\nu \; |\iota_{X^2_t}\alpha| \leq C \cdot l_p(\{\phi_t\}),
\end{equation}

for $C = 2 C_* \cdot C_p.$
%
%

This finishes the proof. It remains to prove Lemma \ref{Lemma: estimate an integral}.

\vspace{2mm}

\begin{pf}(Lemma \ref{Lemma: estimate an integral})

We will prove the equivalent statement that \[\psi_0(a) =\int_{\C} \frac{1}{|z-a|}\frac{1}{(1+|z|^2)^2} dm(z) \leq \frac{C_*}{(|a|^2+1)^{1/2}},\] for a universal constant $C_*$.

Note that this statement is equivalent to the existence of two universal constants $C_1>0$ and $C_2>0,$ such that \[\psi_0(a)\leq C_1 \; \textrm{for} \; |a| \leq 1,\] and \[\psi_0(a) \leq \frac{C_2}{|a|} \; \textrm{for} \; |a| \geq 1.\]

To prove the first statement assume $|a| \leq 1.$ Write the integral $\psi_0(a)$ as the sum of the integrals over the measure-disjoint subsets $\{|z| \leq 2\}$ and $\{|z| \geq 2 \}$ of $\C.$ Then we estimate

\[\int_{|z| \leq 2} \frac{1}{|z-a|}\frac{1}{(1+|z|^2)^2} dm(z) \leq \int_{|z| \leq 2} \frac{1}{|z-a|} dm(z) \leq \]

\[ \leq \int_{|z-a| \leq 3} \frac{1}{|z-a|} dm(z-a) = 6\cdot \pi,\]

and

\[\int_{|z| \geq 2} \frac{1}{|z-a|}\frac{1}{(1+|z|^2)^2} dm(z) \leq \int_{|z| \geq 2} \frac{1}{|z|-|a|}\frac{1}{(1+|z|^2)^2} dm(z) \leq \]

\[ \leq \int_{|z| \geq 2} \frac{2}{|z|}\frac{1}{(1+|z|^2)^2} dm(z) = C'_{1} < \infty. \]
This gives us $\psi_0(a) \leq C_1,$ for $C_1 = 6\cdot \pi + C'_1$.

To prove the second statement assume $|a| \geq 1.$  Write the integral $\psi_0(a)$ as the sum of the integrals over the measure-disjoint subsets
$\{|z - a| \leq \frac{|a|}{2}\}$ and $\{|z-a| \geq \frac{|a|}{2} \}$ of $\C$. Note that if $|z - a| \leq \frac{|a|}{2}$,
then $|z| \geq \frac{|a|}{2}.$ Hence we estimate

\[\int_{|z - a| \leq \frac{|a|}{2}} \frac{1}{|z-a|}\frac{1}{(1+|z|^2)^2} dm(z) \leq \int_{|z - a| \leq \frac{|a|}{2}} \frac{1}{|z-a|} \frac{1}{(1+\frac{1}{4}|a|^2)^2} dm(z) = \]
\[ =  \pi \frac{ |a|}{(1+\frac{1}{4}|a|^2)^2 } \leq  \frac{C'_{2}}{|a|^3}\]
and
\[\int_{|z - a| \geq \frac{|a|}{2}} \frac{1}{|z-a|}\frac{1}{(1+|z|^2)^2} dm(z) \leq \int_{|z - a| \leq \frac{|a|}{2}} \frac{2}{|a|}\frac{1}{(1+|z|^2)^2} dm(z) \leq\]
\[ \leq \int_\C \frac{2}{|a|}\frac{1}{(1+|z|^2)^2} dm(z) = \frac{2\pi}{|a|}.\]
This gives us $\psi_0(a) \leq \frac{C_2}{|a|},$ for $C_2 =C'_2 + 2 \pi$.
This finishes the proof of the lemma.
\end{pf}
\end{pf}

\subsection{Estimating the short path terms}

We now describe the nice choice of short paths.

Recall that the affine chart $U_0$ is given by $u_0: \C \xrightarrow{\sim} U_0 = \C P^1 \setminus \{[1,0]\} \subset \C P^1, z \mapsto [z,1]$.
Consider the subset $X_n(U_0)$ in $X_n(\C P^1)$ of full measure. Choose pairwise distinct base-points
$m_k \in \C$, $1\leq k \leq n$, for example $m_k = \varepsilon \cdot e^{i k \cdot 2\pi/n}$, where $0<\varepsilon$ is small. These points combine to a base-point $m \in X$.

For a point $y=(y_1,...,y_n) \in X_1$, we have in particular that $y_j \in \C$, $1\leq j \leq n$. For each $1\leq j \leq n$ consider the linear path \[\{m_j + t \cdot (y_j - m_j)\}_{t \in [0,1]}\] between $m_j$ and $y_j$ in $\C$. Combine these paths to obtain a path $\gamma(y)$ between $m$ and $y$ in $M^n$. Investigating the $y$ for which the path $\gamma(y)$ is not contained in $X$, we obtain a union $N$ of open subsets in submanifolds (with corners) of codimension $1$. Thus we obtain a continuous system of paths as required on the subset $X' = X_n(U_0) \setminus N$ of full measure. Extend it measurably to $X$. This choice of paths allows us to prove the following.
\vspace{2mm}
\begin{lma}\label{Lemma: bound on angles, nice short paths}
For the nice choice of paths, we have $\int_{\gamma(y)}|\theta_{ij}| \leq \pi$ for all $y \in X'$.
\end{lma}

\begin{pf}
Indeed $\theta_{ij} = p_{ij}^* \theta$, for $\theta = Im(\alpha)$ in the chart $U_0 \times U_0$ given by $\alpha = \frac{d(a-z)}{a-z}$. Hence $\int_{\gamma(y)}|\theta_{ij}|$ equals the total variation of angle of the linear path $\{ a_t - z_t \}_{t \in [0,1]}$ in $\C$ (for $a_t, z_t$ corresponding to $\gamma(y)(t)_i, \gamma(y)(t)_j$ in the given chart), that does not pass through the origin. The bound is now immediate.
\end{pf}

\section{Examples of quasimorphisms and bi-Lipschitz embeddings of vector spaces}\label{Section: examples}

For $\alpha\in P_n = P_n(\C)$ we denote by $\widehat{\alpha}$ the $n$-component link which is a closure of $\alpha$, see Figure \ref{fig:braid-closure1}.
\begin{figure}[htb]
\centerline{\includegraphics[height=1.5in]{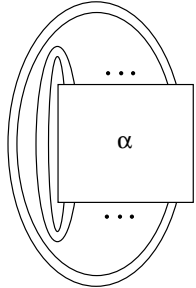}}
\caption{\label{fig:braid-closure1} Closure $\widehat{\alpha}$ of a braid $\alpha$.}
\end{figure}

Let $\sign_n\colon P_n\to \mathbb{Z}$ be a map such that $\sign_n(\alpha)=\sign(\widehat{\alpha})$, where $\sign$ is a signature invariant of links in $\mathbb{R}^3$. Gambaudo-Ghys \cite{GambaudoGhysCommutators} and the first named author \cite{BrandenburskyKnots} showed that $\sign_n$ defines a quasimorphism on $P_n$. We denote by $\overline{\sign}_n\colon P_n\to \mathbb{R}$ the induced homogeneous quasimorphism. Recall that the center of $P_n$ is isomorphic to $\mathbb{Z}$. Let $\Delta_n$ be a generator of the center of $P_n$. It is a well known fact that $P_n(S^2)$ is isomorphic to the quotient of $P_{n-1}$ by the cyclic group $\langle\Delta_{n-1}^2\rangle$, see \cite{Bir}. Let $\lk_n\colon P_n\to \mathbb{Z}$ be a restriction to $P_n$ of a canonical homomorphism from $B_n=B_n(\C)$ to $\mathbb{Z}$ which takes value $1$ on each Artin generator of $B_n$. Let $s_{n-1}\colon P_{n-1}\to \mathbb{R}$ be a homogeneous quasimorphism defined by
$$s_{n-1}(\alpha):=\overline{\sign}_{n-1}(\alpha)-\frac{\overline{\sign}_{n-1}(\Delta_{n-1})}{\lk_{n-1}(\Delta_{n-1})}\lk_{n-1}(\alpha).$$
Since $s_{n-1}(\Delta_{n-1})=0$, the homogeneous quasimorphism $s_{n-1}$ descends to a homogeneous quasimorphism
$\overline{s}_n\colon P_n(S^2)\to \mathbb{R}$. Note that $\overline{s}_2$ and $\overline{s}_3$ are trivial because $P_2(S^2)$ and $P_3(S^2)$ are finite groups.

For each $n\geq 4$ let
$$\overline{\Sign}_n\colon \Diff(S^2,dV)\to \mathbb{R}$$
be the induced homogeneous quasimorphism.
In \cite[Section 5.3]{GambaudoGhysCommutators} Gambaudo-Ghys evaluated quasimorphisms $\overline{\Sign}_{2n}$ on a family of diffeomorphisms $$f_\o\colon S^2\to S^2,$$
such that $f_\o(\infty)=\infty$ and $f_\o(x)=e^{2i\pi\o(|x|)}x$, here $S^2$ is identified with $\mathbb{C\cup\{\infty\}}$, and
$\o\colon \mathbb{R}_+\to\mathbb{R}$ is a function which is constant in a neighborhood of $0$ and outside
some compact set. Let $a(r)$ be the spherical area (with the normalization $\vol(\C) = 1$) of the disc in $\mathbb{C}$ with radius $r$ centered at $0$. Set $u=1-2a(r)$ and let $\widetilde{\o}(u)=\o(r)$. In \cite[Lemma 5.3]{GambaudoGhysCommutators} Gambaudo-Ghys showed that for each $n\geq 2$
\begin{equation}\label{eq:GG-sign-comp}
\overline{\Sign}_{2n}(f_\o)=\frac{n}{2}\int\limits_{-1}^1(u^{2n-1}-u)\widetilde{\o}(u)du.
\end{equation}


\begin{proof}[Proof of Corollary \ref{Corollary: bi-Lipschitz lattices}]
Let $H_\o\colon S^2\to \mathbb{R}$ be a smooth function supported away from the $\{\infty\}$ point and $f_{t,\o}$ be a Hamiltonian flow generated by $H_\o$, such that $f_{1,\o}=f_\o$. Since $f_{t,\o}$ is an autonomous flow, by \eqref{eq:GG-sign-comp} we have
$$
\overline{\Sign}_{2n}(f_{t,\o})=t\frac{n}{2}\int\limits_{-1}^1(u^{2n-1}-u)\widetilde{\o}(u)du.
$$

Let $d\in \mathbb{N}$. It follows from \eqref{eq:GG-sign-comp} that it is straight forward to construct a family of functions $\o_i\colon \mathbb{R}_+\to\mathbb{R}$ and $\{H_{\o_i}\}_{i=1}^d$ supported away from the $\{\infty\}$ point such that
\begin{itemize}
\item
Each Hamiltonian flow $f_{t,\o_i}$ is generated by $H_{\o_i}$ and $f_{1,\o_i}=f_{\o_i}$.
\item
The functions $\{H_{\o_i}\}_{i=1}^d$ have disjoint support and hence the diffeomorphisms $f_{t,\o_i}$ and $f_{s,\o_j}$ commute for all $s,t\in \mathbb{R}$, $1\leq i,j\leq n$.
\item
The $(d\times d)$ matrix
$\left(
              \begin{array}{ccc}
                \overline{\Sign}_{4}(f_{1,\o_1}) & \cdots & \overline{\Sign}_{4}(f_{1,\o_d}) \\
                \vdots & \vdots & \vdots \\
                \overline{\Sign}_{2d+2}(f_{1,\o_1}) & \cdots & \overline{\Sign}_{2d+2}(f_{1,\o_d}) \\
              \end{array}
            \right)
$
is non-singular.
\end{itemize}

It follows that there exists a family $\{\overline{\Phi}_i\}_{i=1}^d$ of homogeneous quasimorphisms on $\Diff(S^2,dV)$, such that $\overline{\Phi}_i$ is a linear combination of $\overline{\Sign}_{4},\ldots,\overline{\Sign}_{2d+2}$ and
\begin{equation}\label{eq:sign-basis}
\overline{\Phi}_i(f_{t,\o_j})=\left\{
                                          \begin{array}{c}\begin{aligned}
                                            &t  &\rm{if}&\quad i=j\\
                                            &0  &\rm{if}&\quad i\neq j\\
                                            \end{aligned}
                                          \end{array}
                                        \right ..
\end{equation}

Let $I\colon\mathbb{R}^d\to \Diff(S^2,dV)$ be a map, such that
$$I(v):=f_{v_1,\o_1}\circ\ldots\circ f_{v_d,\o_d}$$
and $v=(v_1,\ldots,v_d)$. It follows from the construction of $\{f_{v_i,\o_i}\}_{i=1}^d$ that $I$ is a monomorphism. Let $A':=\max\limits_i{\|f_{1,\o_i}\|_p}$, where $\|\cdot\|_p$ denotes the $L^p$-norm, then
$$\|f_{v_1,\o_1}\circ\ldots\circ f_{v_d,\o_d}\|_p\leq A'\|v\|,$$
where $\|v\|=\sum\limits_{i=1}^d |v_i|$.

All diffeomorphisms $f_{v_1,\o_1},\ldots,f_{v_d,\o_d}$ pair-wise commute. Hence, for each $1\leq i\leq d$, by our main Theorem and \eqref{eq:sign-basis} we have
$$\|f_{v_1,\o_1}\circ\ldots\circ f_{v_d,\o_d}\|_p\geq A^{-1}\left|\overline{\Phi}_i(f_{v_1,\o_1}\circ\ldots\circ f_{v_d,\o_d})\right|=
A^{-1}\cdot |v_i|\left|\overline{\Phi}_i(f_{1,\o_i})\right|,$$
where $A$ is the maximum over the Lipschitz constants (in our main theorem) of the functions
$$\overline{\Phi}_i\colon\Diff(S^2,dV)\to \mathbb{R}.$$
It follows that
$$\|f_{v_1,\o_1}\circ\ldots\circ f_{v_d,\o_d}\|_p\geq \left((d\cdot A)^{-1}\min_i\left|\overline{\Phi}_i(f_{1,\o_i})\right|\right) \|v\|=
\left((d\cdot A)^{-1}\right) \|v\|,$$
and the proof follows.
\end{proof}

\appendix
\section{Integrability of the quasimorphisms}\label{Appendix: Integrability}

For every pair of points $x,y\in S^2$ let us choose a minimal geodesic path $s_{xy}\colon[0,1]\to S^2$ from $x$ to $y$.
Let $f_t\in\Diff(S^2,dV)$ be an isotopy from the identity to $f\in\Diff(S^2,dV)$
and let $m\in S^2$ be a basepoint. For $y\in S^2$ we define a loop
$\gamma_{y}\colon [0,1]\to S^2$ by
\begin{equation}\label{eq:gamma-for-qm}
\gamma_{y}(t):=
\begin{cases}
s_{my}(3t) &\text{ for } t\in \left [0,\frac13\right ]\\
f_{3t-1}(y) &\text{ for } t\in \left [\frac13,\frac23\right ]\\
s_{f(y)m}(3t-2) & \text{ for } t\in \left [\frac23,1\right ].
\end{cases}
\end{equation}

Let $X_n(S^2)$ be the configuration space of all ordered $n$-tuples
of pairwise distinct points in the surface $S^2$. It's fundamental group
$\pi_1(X_n(S^2))$ is identified with the pure braid group $P_n(S^2)$.
Fix a base-point $m=(m_1,\ldots,m_n)$ in $X_n(S^2)$. For almost each $x=(x_1,\ldots,x_n)\in X_n(S^2)$ the $n$-tuple of loops $(\gamma_{x_1},\ldots,\gamma_{x_n})$ is
a based loop in the configuration space $X_n(S^2)$.
Let $\gamma(f,x)\in P_n(S^2)$
be an element represented by this loop.

Let $r\colon P_n(S^2)\to \mathbb{R}$ be a homogeneous quasimorphism. Note that since $\pi_1(\Diff(S^2,dV))\cong\mathbb{Z}/2\mathbb{Z}$, the value $r(\gamma(f;{x}))$ does not depend on the isotopy $f_t$. Define $\Phi\colon \Diff(S^2,dV)\to \mathbb{R}$ by
\begin{equation}\label{eq:GG-ext}
\Phi(f):=
\int\limits_{X_n(S^2)}r(\gamma(f;{x}))d{x}\qquad\qquad
\overline{\Phi}(f):=\lim_{k\to +\infty}\frac{\Phi(f^k)}{k}
\thinspace .
\end{equation}

\begin{lma}
The function $\overline{\Phi}\colon\Diff(S^2,dV)\to \mathbb{R}$ is a well defined homogeneous quasimorphism.
\end{lma}
\begin{proof}

It was proved by Banyaga that there exists $d\in\mathbb{N}$ and a family $\{f_i\}_{i=1}^d$ of diffeomorphisms in the group $\Diff(S^2,dV)$, such that each $f_i$ is supported in some disc $D_i\subset S^2$ and $f=f_1\circ\ldots\circ f_d$, see \cite{BanyagaStructure}. For each
$1\leq i\leq d$ pick an isotopy $\{f_{t,i}\}$ in $\Diff(S^2,dV)$ between the identity and $f_i$, such that the support of $\{f_{t,i}\}$ lies in the same disc $D_i$. First we are going to show that $|\Phi(f_i)|<\infty$.

Let $X_n(D_i)$ be the configuration space of all ordered $n$-tuples
of pairwise distinct points in the disc $D_i\subset S^2$.
Since changing a basepoint in $X_n(S^2)$ changes $\Phi$ by a bounded value, we can assume that the basepoint $m$ lies in $X_n(D_i)$. It follows that
\begin{eqnarray*}
\Phi(f_i)&=&\int\limits_{X_n(S^2)}r(\gamma(f_i;{x}))\,d{x}\\
&=&\sum_{j=1}^n\dbinom{n}{j}\vol(X_{n-j}(S^2\setminus D_i))\int\limits_{X_j(D_i)}r(\gamma(f_i;{x}))\,d{x}.
\end{eqnarray*}
Since $r\colon P_n(S^2)\to \mathbb{R}$ is a homogeneous quasimorphism, there exists a constant $C>0$, such that
$$|r(\gamma)|\leq C |\gamma|_{P_n(S^2)},$$
where $|\gamma|_{P_n(S^2)}$ is the word length of $\gamma$ with respect to the Artin generating set of $P_n(S^2)$. The support of the isotopy $\{f_{t,i}\}$ lies in the disc, hence the representative of the braid $\gamma(f_i;{x})$, which was built using $\{f_{t,i}\}$, defines a braid in $P_n$. It is a well-known fact that $P_n(S^2)$ is a factor group of $P_{n-1}$. It follows that $|\gamma(f_i;{x})|_{P_n(S^2)}\leq |\gamma(f_i;{x})|_{P_n}$, where $|-|_{P_n}$ is the word length in $P_n$ with respect to the Artin generating set of $P_n$. For $1\leq j<n$ the group $P_j$ may be viewed as a subgroup of $P_n$ by adding $n-j$ strings. It follows that
\begin{equation*}
|\Phi(f_i)|\leq C\sum_{j=1}^n\dbinom{n}{j}\vol(X_{n-j}(S^2\setminus D_i))\int\limits_{X_j(D_i)}|\gamma(f_i;{x})|_{P_n}\,d{x}.
\end{equation*}

The integral $\int\limits_{X_j(D_i)}|\gamma(f_i;{x})|_{P_n} \,d{x}$ is well defined for each $1\leq j\leq n$, see \cite[Lemma 4.1]{BrandenburskyKnots}, hence  $$|\Phi(f_i)|<\infty.$$
Let $g',h'\in \Diff(S^2,dV)$. Then
\begin{eqnarray*}
&&|\Phi(g'h')-\Phi(g')-\Phi(h')|\\
&\leq&\int\limits_{X_n(S^2)}|r(\gamma(g'h';{x}))-r(\gamma(g';h'(x)))-r(\gamma(h';{x}))|\,d{x}\\
&\leq&\vol(X_n(S^2))\cdot \delta(r)\thinspace,
\end{eqnarray*}
i.e. $\Phi$ satisfies the quasimorphism condition. It follows that
$$|\Phi(f)|\leq d\left(\delta(r)\cdot\vol(X_n(S^2))+\sum_{i=1}^d|\Phi(f_i)|\right).$$
Hence $\Phi$ is a well defined quasimorphism and so is $\overline{\Phi}$.
\end{proof}

\bibliographystyle{amsplain}
\bibliography{LpMetricTwoSphereRefs}

\vspace{3mm}

\textsc{Michael Brandenbusrky, Department of Mathematics, University of Montreal,
CP 6128, Succ. Centre-Ville Montr\'{e}al, QC H3C 3J7}\\
\emph{E-mail address:} \verb"michael.brandenbursky@mcgill.ca"

\vspace{3mm}

\textsc{Egor Shelukhin, Department of Mathematics, University of Montreal,
CP 6128, Succ. Centre-Ville Montr\'{e}al, QC H3C 3J7}\\
\emph{E-mail address:} \verb"shelukhi@crm.umontreal.ca"

--------------------------------------------------------------------------------

\end{document}